\documentclass[10pt]{article}
\font\smallit=cmti10

\usepackage{amssymb,latexsym,amsmath,epsfig,amsthm} 

\usepackage{amsfonts}
\usepackage{hyperref}

\usepackage{graphicx}
\usepackage{float}
\usepackage{young}
\usepackage[vcentermath]{youngtab}
\usepackage{enumerate}

\usepackage{cleveref}
\usepackage{mathrsfs}

\hypersetup{colorlinks=true,linkcolor=blue,citecolor=magenta}

\makeatletter

\renewcommand\section{\@startsection {section}{1}{\z@}
{-30pt \@plus -1ex \@minus -.2ex}
{2.3ex \@plus.2ex}
{\normalfont\normalsize\bfseries\boldmath}}

\renewcommand\subsection{\@startsection{subsection}{2}{\z@}
{-3.25ex\@plus -1ex \@minus -.2ex}
{1.5ex \@plus .2ex}
{\normalfont\normalsize\bfseries\boldmath}}

\renewcommand{\@seccntformat}[1]{\csname the#1\endcsname. }

\makeatother

\newtheorem{theorem}{Theorem}

\newtheorem{corollary}[theorem]{Corollary}

\theoremstyle{definition}
\newtheorem{definition}[theorem]{Definition}
\newtheorem{remark}[theorem]{Remark}

\theoremstyle{plain}

\begin{document}

\begin{center}
\uppercase{\bf The product of parts or ``norm'' of a partition}
\vskip 20pt
{\bf Robert Schneider
}\\
{\smallit Department of Mathematics, University of Georgia, Athens, Georgia, USA}\\
{\tt robert.schneider@uga.edu}\\ 
\vskip 10pt
{\bf Andrew V. Sills
}\\
{\smallit Department of Mathematical Sciences, Georgia Southern University, 
Statesboro and Savannah, Georgia, USA}\\
{\tt asills@georgiasouthern.edu}\\ 
\end{center}
\vskip 30pt


\centerline{\bf Abstract}
\noindent
In this article we study the ``norm'' of an integer partition, which we define to be the product of the parts. This partition-theoretic statistic has appeared here and there in the literature of the last century or so, and is at the heart of current research by both authors. We survey known results and give new results related to this all-but-overlooked object, which, it turns out, plays a comparable role in partition theory to the size, length, and other standard partition statistics.  

\pagestyle{myheadings} 
\thispagestyle{empty} 
\baselineskip=12.875pt 
\vskip 30pt

\newcommand{\Mu}{\mathrm{M}}
\newcommand{\PP}{\mathbb{P}}
\newcommand{\NN}{\mathbb{N}}
\newcommand{\ldb}{\{\hskip-1.2mm\{ }
\newcommand{\rdb}{ \} \hskip-1.2mm\} }
\newcommand{\vv}{\mathbf{v}}
\newcommand{\rr}{\mathbf{R}}
\newcommand{\ee}{\mathbf{e}}
\newcommand{\N}{\mathbb{N}}
\renewcommand{\P}{\mathbb{P}}
\newcommand{\Z}{\mathbb{Z}}
\newcommand{\Q}{\mathbb{Q}}
\newcommand{\HH}{\mathbb{H}}
\newcommand{\R}{\mathbb{R}}
\newcommand{\C}{\mathbb{C}}
\newcommand{\A}{\mathbb{A}}
\newcommand{\D}{\mathbb{D}}
\newcommand{\F}{\mathcal{F}}
\newcommand{\Rsym}{\mathbb{R}^{n\times n}_{sym}}
\newcommand{\bbH}{\mathbb{H}}
\newcommand{\calO}{\mathcal{O}}
\newcommand{\mymod}{\operatorname{mod}}
\newcommand{\ggT}{\operatorname{ggT}}
\newcommand{\Hom}{\operatorname{Hom}}
\newcommand{\End}{\operatorname{End}}
\newcommand{\Ends}{\End_{sym}}
\newcommand{\Aut}{\operatorname{Aut}}
\newcommand{\Spur}{\textnormal{Spur}}
\newcommand{\Grad}{\operatorname{Grad}}
\newcommand{\Kern}{\textnormal{Kern}}
\newcommand{\Stab}{\operatorname{Stab}}
\newcommand{\eps}{\varepsilon}
\newcommand{\GL}{\operatorname{GL}}
\newcommand{\Sp}{\operatorname{Sp}}
\newcommand{\SL}{\operatorname{SL}}
\newcommand{\PGL}{\operatorname{PGL}}
\newcommand{\PSL}{\operatorname{PSL}}
\newcommand{\G}{\mathscr{G}}
\renewcommand{\d}{\mathrm{d}}
\newcommand{\M}{\mathcal{M}}
\newcommand{\sign}{\operatorname{sgn}}
\renewcommand{\Re}{\operatorname{Re}}
\renewcommand{\Im}{\operatorname{Im}}
\newcommand{\calR}{\mathscr{R}}
\newcommand{\SLZ}{\SL_2(\Z)}
\newcommand{\abcd}{\left(\begin{smallmatrix} a & b \\ c & d \end{smallmatrix}\right)}
\newcommand{\bigabcd}{\begin{pmatrix} a & b \\ c & d \end{pmatrix}}
\newcommand{\Res}{\operatorname{Res}}
\newcommand{\calF}{\math{F}}
\newcommand{\spt}{\operatorname{spt}}
\newcommand{\calS}{\mathcal{S}}
\newcommand{\pihol}{\pi_{hol}}
\newcommand{\tr}{\operatorname{trace}}
\newcommand{\fkD}{f_{k,D}}
\newcommand{\calQ}{\mathcal{Q}}
\newcommand{\calE}{\mathcal{E}}
\newcommand{\calP}{\mathscr{P}}
\newcommand{\rank}{\operatorname{rank}}
\newcommand{\prem}{\operatorname{prem}}
\newcommand{\calD}{\mathcal{D}}
\newcommand{\z}{\tau}

\section{Introduction: a multiplicative statistic on (additive) partitions}
The theory of {integer partitions} is a rich source of identities, bijections, and 
interrelations at the confluence of number theory, combinatorics, algebra, analysis, and 
the physical sciences.  Let 
$$\lambda= (\lambda_1, \lambda_2 , \dots , \lambda_r)$$
denote a generic partition, with integer parts $\lambda_1 \geq \lambda_2  \geq \dots 
\geq \lambda_r \geq 1$, and let $\emptyset$ denote the {\it empty partition}. 
Alternatively, it is often useful to notate partitions using classical ``frequency superscript 
notation'', viz., 
$$\lambda = \left< 1^{m_1}\  2^{m_2}\  3^{m_3}\cdots\right>$$
where $m_j=m_j(\lambda)$ is the {\it frequency} of occurrence (or {\it multiplicity}) of $j$ 
as a part in the partition, noting only finitely many $m_j$ are nonzero, with the 
conventions that if $m_j=1$ then it may 
be omitted in the superscript, and $j^{m_j}$ is usually omitted if $m_j=0$.

Many famous identities are related to the statistic $p(n)$ called the {\it partition 
function}, counting the number of partitions of a natural number $n$, like Euler's seminal 
{\it partition generating function}.
\begin{theorem}[Euler]\label{pgen}
For $q\in \mathbb C, |q|<1$ we have that
$$\prod_{n=1}^{\infty}\frac{1}{1-q^n}=\sum_{n=0}^{\infty}p(n)q^n.$$
\end{theorem} 
Other statistics about partitions also feature heavily into partition theory, such as the 
\emph{size} $|\lambda|:=\lambda_1+\lambda_2+\cdots+\lambda_r$ of partition $
\lambda$ (sum of the parts), the {\it length} $\ell(\lambda):= r$ of $\lambda$ (number of 
parts), the {\it largest part} $\text{lg}(\lambda):=\lambda_1$, Dyson's {\it rank} 
$rk(\lambda):=\text{lg}(\lambda)-\ell(\lambda)$, etc. 

Here we will study another, all-but-overlooked statistic that plays a comparable role in 
partition theory to the size, length, and others listed above. 
\begin{definition}
Let $N(\lambda)$ 
denote the product of the parts, or \emph{norm}, of the partition $\lambda$:
$$N(\lambda):=\lambda_1 \lambda_2 \lambda_3 \cdots \lambda_r,$$
with $N(\emptyset):=1$ (it is an empty product). Equivalently, we have $N(\lambda)= 
1^{m_1}2^{m_2}3^{m_3}\cdots$.
\end{definition}

The defining characteristic of the set $\mathcal P$ of partitions is that one {\it adds} the 
parts together, so this multiplicative norm perhaps feels a little artificial. On the other 
hand, if we view a partition purely as a multiset of whole numbers, then multiplying the 
elements together is just as natural an operation as adding them. 
Likewise, one can express the size $|\lambda|$ in terms of the norm:
\begin{equation}\label{eq1}
  |\lambda|=N(\lambda)\sum_{\lambda_i \in \lambda}\frac{1}{N\left(\lambda/\lambda_i
  \right)},\end{equation}
where ``$\lambda_i \in \lambda$'' indicates $\lambda_i \in \mathbb N$ is a part of $
\lambda$, and we let $\lambda/\lambda_i\in \mathcal P$ denote the partition obtained 
by deleting $\lambda_i$ from $\lambda$. (This identity follows instantly from considering 
the ratio ${|\lambda|}/{N(\lambda)}$.) 

MacMahon's partial fraction decomposition of the generating function for partitions of 
length at most $n$ may be the first explicit appearance of the partition norm in the 
literature (notated below in the conventions of this paper) \cite{M1,M2}.
\begin{theorem}[MacMahon]\label{MacMahon}
For $q$ not equal to a $k$th root of unity, $1\leq k\leq n$, we have that
$$\prod_{j=1}^{n}\frac{1}{1-q^j}=\sum_{\lambda \vdash n}\frac{1}{N(\lambda)\  m_1!\  
m_2!\  m_3!\cdots  (1-q)^{m_1}(1-q^2)^{m_2}(1-q^3)^{m_3}\cdots}$$
where ``$\lambda \vdash n$'' on the right side means the sum is taken over the 
partitions of $n$.\end{theorem}

The partition norm
features centrally in the first author's work (e.g., \cite{Robert_zeta, Robert_bracket, 
PhD}) on partition zeta functions and partition analogs of classical arithmetic functions, 
and the second author independently studied the product of parts in his own work 
\cite{AS} 
on MacMahon's partial fractions.

Immediately, there are a number of questions one might ask about this partition statistic. 
For example, does it have a product-sum generating function interpretation? Does the 
norm admit a natural combinatorial (or probabilistic) interpretation? What are its 
maximum, minimum and average values over the partitions of $n$? Does the norm 
obey any nice asymptotics? 
Does it connect to other areas of partition theory or, more broadly, of mathematics? 

\section{Generating functions and dotted Young diagrams}

Let us note that a generating function $\sum_{\nu} P(\nu)q^{\nu}$, where $P(\nu)$ is the 
number of partitions of norm $\nu$, is not possible, as there are infinitely many 
partitions of any fixed norm $\nu\geq 1$: adjoining arbitrarily many $1$'s to a partition 
gives a new partition of the same norm. Moreover, we cannot control multiplication in 
the exponent of $q$ via product generating functions in the same way we generate 
partitions in the exponent. We can get close, though, if we consider norms of partitions 
with no $1$'s and relax our expectations for a power series generating function.

Following \cite[Appendix A]{PhD}, let a \emph{nuclear partition} be a partition in which 
all parts are greater than $1$ (thus $|\mu|\leq N(\mu)$ for $\mu$ a nuclear partition). 
Then the (finite) number of nuclear partitions of fixed norm $\nu$ (which is equivalent to 
the number of multiplicative partitions of $\nu$) has the following ``non-power series'' 
generating function.

\begin{theorem}\label{loggen}
Let $\widetilde{P}(\nu)$ denote the number of nuclear partitions of fixed norm $\nu\geq 
1$. Then for $x\in \mathbb R,\  0< x < e^{-1},$ we have
\begin{equation} \label{GF}  
\prod_{n = 2}^{\infty} \frac{1}{1 - x^{\log n}}=\sum_{\nu=1}^{\infty}\widetilde{P}(\nu)x^{\log 
\nu}. \end{equation}
\end{theorem}

\begin{proof}
Observing for any partition $\lambda$ that $\log \lambda_1 + \log \lambda_2 + \cdots 
\log \lambda_r = \log N(\lambda)$, then as the product starts with index $n=2$, 
classical 
generating function ideas yield the identity. For justification that the product and sum 
converge for $0<x<e^{-1}$, we refer the reader to the proof of Theorem \ref{normthm} 
below.
\end{proof}

We now offer a combinatorial interpretation of $N(\lambda)$. Recall the {\it Young 
diagram} for a partition $(\lambda_1, \lambda_2, \dots, \lambda_r)$, in which the $i$th 
part is pictured as the $i$th row of the diagram consisting of $\lambda_i$ squares, e.g., 
the Young diagram for $\lambda=(4,3,3,1)$ is:

$$\begin{Young}
&     & &\cr
&  &  \cr
   & &\cr
   \cr
\end{Young}
$$
Let us impose further structure on this diagram by placing a dot in one of the squares 
of 
each horizontal row, and call the resulting diagram a {\it dotted Young diagram} of 
partition $\lambda$:
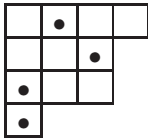
$$\begin{Young}
&    \textbullet & &\cr
&  &   \textbullet \cr
 \textbullet & &\cr
\textbullet  \cr
\end{Young}
$$
This pattern of dots is not unique; here is another dotted Young diagram of $\lambda$:

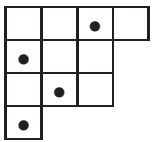
$$\begin{Young}
&  &  \textbullet  &\cr
 \textbullet&  &   \cr
  & \textbullet&\cr
\textbullet  \cr
\end{Young}
$$
The different dot patterns for a given Young diagram are enumerated by the norm.

\begin{theorem}\label{young}
The number of dotted Young diagrams of a partition $\lambda$ is $N(\lambda)$.
\end{theorem}

\begin{proof}
There are $\lambda_1$ different ways to place a dot in row one, $\lambda_2$ ways to 
dot row two, $\lambda_3$ ways to dot row three, etc., yielding $\lambda_1 \lambda_2 
\lambda_3 \cdots \lambda_r=N(\lambda)$ different dotted Young diagrams of $\lambda
$.
\end{proof}

\begin{remark}
More generally, if we place $k$ dots in each row, the number of $k$-tuple dotted Young 
diagrams of $\lambda=\left< 1^{m_1} 2^{m_2}\cdots i^{m_i}\cdots\right>$ is $
\prod_{i=1}
^{\infty}\binom{i}{k}^{m_i},$
with binomial coefficients $:=0$ when $i<k$.
\end{remark}

In the context of dotted Young diagrams, the norm admits the following generating 
function interpretation.

\begin{theorem}
Let $\dot{p}(n)$ denote the number of dotted Young diagrams of size $n$. Then
$$\dot{p}(n)=\sum_{\lambda \vdash n} N(\lambda).$$
For $|q|<1$ we have the generating function
$$\prod_{n=1}^{\infty}\frac{1}{1-nq^n}=\sum_{n=0}^{\infty}\dot{p}(n)q^n.$$
\end{theorem}
\begin{proof}
 The first equation of the theorem is an immediate corollary of Theorem \ref{young}. 
 The generating function for $\dot{p}(n)$ follows naturally from this corollary together 
 with \cite[Corollary 4.3]{Robert_bracket}:
 \begin{equation}
 \prod_{n=1}^{\infty}\frac{1}{1-nq^n}=\sum_{\lambda \in \mathcal P}N(\lambda)q^{|
 \lambda|}=\sum_{n=0}^{\infty}q^n \sum_{\lambda \vdash n} N(\lambda).
 \end{equation}
\end{proof}
We may define a yet more general object. For a fixed dotted Young diagram of $
\lambda$, if $i$ appears as a part with frequency $m_i>1$, we will color each of the 
dots differently over the $m_i$ rows of $i$ squares (that is, we give each dot one of 
$m_i$ distinct colors). Let us call such a diagram a {\it multicolor dotted Young 
diagram} of partition $\lambda$.

Here are two different colorings of the same dotted Young diagram of $
\lambda=(5,5,3,3,3,1)$:

$$\begin{Young}
&    \textcolor{blue}{\textbullet}
 & & &\cr
&   
 & & \textcolor{red}{\textbullet}&\cr
&  &  \textcolor{blue}{\textbullet}
 \cr
 \textcolor{red}{\textbullet}&  & 
 \cr
 &\textcolor{green}{\textbullet}
  &\cr
\textbullet  \cr
\end{Young}
 \  \  \  \  \  \  \  \  \  \  \  \  \  \  \  \  \  \  \  \  \  \  \  \  \  \  \  \  \  \  \  \  \  
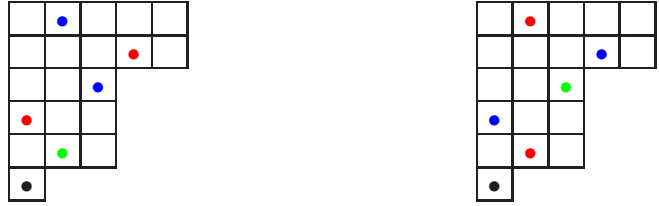
 \begin{Young}
&    \textcolor{red}{\textbullet}
 & & &\cr
&   
 & & \textcolor{blue}{\textbullet}&\cr
&  &  \textcolor{green}{\textbullet}
 \cr
 \textcolor{blue}{\textbullet}&  & 
 \cr
 &\textcolor{red}{\textbullet}
  &\cr
\textbullet  \cr
\end{Young}
$$

\begin{theorem}\label{multicolor}
The number of multicolor dotted Young diagrams of a partition $\lambda$ is 
$$N(\lambda)\  m_1!\   m_2!\   m_3! \cdots m_i! \cdots.$$
\end{theorem}

\begin{proof}
There are $N(\lambda)$ different dotted Young diagrams of $\lambda$, and $m_i!$ 
ways to permute $m_i$ colors among the rows of length $i$ in each dotted diagram. 
\end{proof}

Thus the probability of picking a particular multicolor dotted Young diagram of a fixed 
partition $\lambda$ is 
\begin{equation}
\frac{1}{N(\lambda)\  m_1!\   m_2!\   m_3! \cdots m_i! \cdots}.
\end{equation}

In \cite{AS} the second author refers to these fractions as {\it MacMahon coefficients} 
of the partial fraction decomposition in Theorem \ref{MacMahon}, and
in~\cite{AVS_MMstats} shows by the following result of N. J. Fine~\cite[p. 38, Eq. 
(22.2)]{F} that, 
if each partition of $n$ occurs with the probability equal to its MacMahon coefficient,
 this is a discrete probability distribution. 
 
 \begin{theorem}[Fine] We have that
\[
  \sum_{\lambda\vdash n} \frac{1}
  {N(\lambda) \  m_1! \  m_2! \  m_3!\cdots}=1.\]
\end{theorem} 

    This identity can be viewed as the $q=0$ case of Theorem \ref{MacMahon}. 
Numerous 
identities involving MacMahon coefficients arise naturally from the classical Fa\`{a} di 
Bruno's formula (see, e.g., \cite[Appendix D]{PhD}), like the following result stemming 
from 
Euler's partition generating function. 

\begin{theorem}\label{norm2}
Let $p(n)$ denote the partition function. Then we have $$p(n)=\sum_{\lambda\vdash n}
\frac{\sigma(1)^{m_1}\sigma(2)^{m_2}\sigma(3)^{m_3}\cdots}
  {N(\lambda) \  m_1! \  m_2! \  m_3!\cdots},$$
where $\sigma(n):=\sum_{d|n}d$. 
 \end{theorem} 

\begin{proof} Setting $a(n)\equiv c$ identically, $c \geq 0,$ in  Prop. D.1.1 of \cite{PhD} 
gives for $|q|<1$ the formula 
$$\prod_{n=1}^{\infty}\frac{1}{(1-q^n)^{\pm c}} = \sum_{n=0}^{\infty} q^n \sum_{\lambda
\vdash n}(\pm c)^{\ell(\lambda)}
\frac{\sigma(1)^{m_1}\sigma(2)^{m_2}\sigma(3)^{m_3}\cdots}
  {N(\lambda) \  m_1! \  m_2! \  m_3!\cdots},$$
where ``$\pm$'' represents the same sign, positive or negative, on both the left- and 
right-hand side. Letting $c=1$ with ``$\pm=$ plus'' gives the theorem, by comparison 
with 
Theorem \ref{pgen}. 
\end{proof}
Seen in a certain light, the norm is a component of the partition function $p(n)$. We 
would like to find combinatorial interpretations for formulas like Theorem \ref{norm2} 
arising from Fa\`{a} di Bruno's formula, as well.

\section{Maximum, minimum and average values of the norm}

Another immediate question one asks about a statistic such as the norm is, ``How big 
is 
it?'' Then it is natural that much of the 
literature related to the product of the parts of partitions seems to focus on the 
magnitude 
of the product; we survey some of these results, and record a few of our own.

For instance, the following theorem appears 
as an exercise in a few sources, e.g.,
~\cite[pp. 30--31, 188]{H}, ~\cite[p. 5, prob. 15]{N}; to whom to attribute the result is
unclear.
\begin{theorem}[Halmos, Newman, et al.]  Among all partitions of $n\geq 1$, the 
partition with maximum norm is:
  \begin{enumerate}[i.]
     \item $\langle 3^{n/3} \rangle$ if $n\equiv 0\pmod{3}$,
     \item $\langle 3^{(n-4)/3} 4 \rangle$ as well as $\langle 2^2 3^{(n-4)/3} \rangle$ 
     if $n\equiv 1\pmod{3}$ and $n>1$,
     \item $\langle 2\ 3^{(n-2)/3} \rangle$ if $n\equiv 2 \pmod 3$,
     \item $\langle 1 \rangle$ if $n=1$.
  \end{enumerate}
\end{theorem}

\begin{remark}
The sequence $a(n) = $ ``maximum norm over all partitions of $n$'' is A000792 in 
the OEIS~\cite{oeis}.
\end{remark}

More recently, Do\v{s}li\'{c} \cite[Theorem 4.1]{D} gives an analogous result for 
partitions into odd parts.

\begin{theorem}[Do\v{s}li\'{c}] Among all partitions of $n\geq 3$ into odd parts, the 
partition with maximum norm is:
\begin{enumerate}[i.]
   \item $\langle 3^{n/3} \rangle$ if $n\equiv 0\pmod{3}$,
   \item $\langle 1 \ 3^{(n-1)/3} \rangle$ if $n\equiv 1 \pmod{3}$,
   \item $\langle 3^{(n-5)/3} 5 \rangle$ if $n\equiv 2 \pmod{3}$.
\end{enumerate}

\end{theorem}

Another result \cite[Theorem 3.1]{D} of Do\v{s}li\'{c} handles the partitions into distinct 
parts via a connection to triangular numbers.

\begin{theorem}[Do\v{s}li\'{c}]  Let $T_k := k(k+1)/2$, the $k$th triangular number. 
Among the partitions of $n\geq 2$ into distinct parts, the partition $\Delta^{\text{max}}=
\Delta^{\text{max}}(n)$ with maximum norm is
as follows: given that $n$ can be expressed uniquely as $T_k + j$ for some $-1\leq j 
\leq 
k-2$, then 

\[\Delta^{\text{max}}=(k+1, k, k-1, \dots, k-j+1, k-j-1, k-j-2, \dots, 3,2) ,\] i.e., the 
partition in which the parts are
one copy each of all integers $2$ through
$k+1$ inclusive, with the exception of $k-j$.  The norm of this partition is 
$N(\Delta^{\text{max}})=\frac{(k+1)!}{k-j}.$
\end{theorem}

\begin{remark}
The sequence $a(n) = $ ``maximum norm over partitions of $n$ into distinct parts'' is 
A034893 in OEIS~\cite{oeis}. 
We note further connections exist in the literature between partitions and triangular 
numbers (see, e.g., \cite{Sills, M_Schneider}).
\end{remark}

Based on these examples, it seems that the norms of other interesting subclasses of 
partitions may yield analogous results. Here we give another example, which does not 
seem to 
have appeared previously in the literature; we are interested to identify further such 
subclasses.

Recall that $\lambda = (\lambda_1, \lambda_2, \dots, \lambda_r)$ is a 
\emph{Rogers--Ramanujan partition}
 if $\lambda_i - \lambda_{i+1} \geq 2$ for $i=1,2,
\dots,r-1$ (see \cite{Andrews, RR}).  
\vskip 5mm
 
 \begin{theorem} \label{rrmax}
   Let $D_k := k(k+1),$ and write $n = D_k + j$ where $0\leq j < 2k+2$; also, set $j':=j-k
   $ if $j>k$.
Among all Rogers--Ramanujan partitions of size $n$, let $\rho^{\text{max}}=
\rho^{\text{max}}(n)$ denote
the one of maximum norm. 
\begin{enumerate}[i.]
  \item $\rho^{\text{max}}=(2k, 2k-2, 2k-4, \dots, 6, 4, 2)$ with $N(\rho^{\text{max}})=2^k 
  k!$, if $j=0$,
  \item $\rho^{\text{max}} = (2k+1, 2k-1, 2k-3, \dots,2k-2j+3,2k-2j,2k-2(j+1),\dots, 6, 4, 
  2)$ with $N(\rho^{\text{max}})= \frac{2^{k-2j}(k-j)! (k-j+1)!(2k+2)!}{(k+1)! \left(2(k-j)
  +2\right)!}$, if $1\leq j < k$,
  \item $\rho^{\text{max}} = (2k+1, 2k-1, 2k-3, \dots, 7, 5, 3)$ with 
  $N(\rho^{\text{max}})= \frac{(2k+2)!}{2^{k+1}(k+1)!}$, if $j=k$,

  \item $\rho^{\text{max}} = (2k+2, 2k, 2k-2, \dots,2k-2j'+4,2k-2j'+1,2k-2(j'+1)+1,\dots, 7, 
 5, 3)$ with $N(\rho^{\text{max}})= \frac{\left(2(k-j')+2\right)!(k+1)!}{2^{k-2j'+1}(k-j'+1)!
 ^2}$, if $k < j < 2k,$
  
   \item $\rho^{\text{max}} = (2k+2, 2k, 2k-2, \dots, 8, 6, 4)$ with $N(\rho^{\text{max}})= 
   {2^{k+1}(k+1)!}$, if $j=2k$,

   \item $\rho^{\text{max}} = (2k+3, 2k, 2k-2,2k-4, \dots, 8, 6, 4)$ with 
   $N(\rho^{\text{max}})= {2^{k-1}(2k+3) k!}$, if $j = 2k+1$.

 \end{enumerate}

\end{theorem}

\begin{proof}
The theorem (as well as, morally, the preceding results) follows from the simple fact 
that 
for $a,b,c \in \mathbb N$, if $a < b$ then the magnitude of the product $ab$ is more 
significantly enlarged by increasing the smaller factor $a$ by an additive constant $c
\geq1$ 
than by increasing $b$ by an equal amount, as $(a+c)b=ab+bc >  ab+ac=a(b+c)$. By 
the same 
token, if $a+b=n,\  1<a < b$, and we wish to vary the summands while keeping $n$ 
constant, 
the product $ab$ is increased when we borrow $c<b-a$ from the larger summand $b$ 
to increase 
the smaller summand $a$, and is decreased by borrowing $c< a$ from the smaller 
summand to 
increase the larger, as $(a+c)(b-c)=ab+c(b-a)-c^2 \  \geq \  ab\  \geq \   ab-c(b-a)-
c^2=(a-c)(b+c)$. 

The same holds for the sum versus the product of natural numbers $a_1, a_2, \dots, 
a_r$ such 
that $a_1\geq a_2 \geq \dots \geq a_r>1$: borrowing from larger summands of 
$a_1+a_2+\cdots 
+a_r$ to increase smaller summands (or create smaller summands 
greater than $1$), while maintaining 
their relative ``$\geq$'' ordering, generally increases the product $a_1 a_2 a_3\cdots$ 
and 
the opposite action generally reduces the product. (Summands $a_i=1$ break this rule: 
they 
increase size but fix the norm, viz. $1\cdot b < 1+b$.) Noting that partitions of $n$ 
represent exactly such sums $a_1+a_2+\cdots + a_r=n,$ then any partition of $n$ 
might be 
transformed into a partition of $n$ of greater norm by reducing larger parts to increase 
(or 
create new) smaller parts accordingly, so long as the relative ordering of the existing 
parts is not violated. Restrictions on type of integers used or ordering of the parts (e.g., 
differences of a specified kind) limit the transformations possible. 

If we seek a Rogers--Ramanujan partition (distinct parts with differences 
at least $2$) of 
$n=D_k=k(k+1)=2\cdot T_k$, then by the preceding ``borrowing from larger parts to 
increase 
smaller parts'' principle, it is clear in the partition $\alpha:=(2k, 2k-2, 2k-4, \dots, 6, 
4, 2)$ that no part greater than $2k$ may be increased 
(or a new part created) without violating the 
distinctness and difference restrictions. Moreover, any other allowed partition of 
$n=D_k$ 
must be formed by borrowing from {\it smaller} parts of $\alpha$ to increase other 
summands, 
decreasing the norm from $N(\alpha)$. Thus $\alpha$ is the Rogers-Ramanujan 
partition of 
$D_k$ with greatest norm. 

For a Rogers--Ramanujan partition of $n=D_k+j,\  0<j<2k+2,$ we want to stay as close 
in 
minimal shape to $\alpha$ above by distributing the quantity $j$ between the parts of $
\alpha$. But no part of $\alpha$ may be increased unless the preceding part is first 
increased without violating the order restriction, so by the ``borrowing from larger to 
increase smaller'' rule, the partition of largest norm is achieved by adding $1$ to each 
of 
the largest $j$ parts of $\alpha$ in the case $j< k$, to yield partition $(2k+1, 2k-2+1, 
2k-4+1, \dots,2k-2(j-1)+1,2k-2j,\dots, 6, 4, 2)$ of maximum norm. If $j=k$ we ``use up'' 
all 
$k$ of the 1's by this process to yield partition $\beta:=(2k+1, 2k-2+1, 2k-4+1, \dots,
7,5,3)$. For $k<j<2k+1$, restart the process of adding 1 to each of the largest $j'=j-k$ 
parts of $\beta$ to yield partition $(2k+2, 2k-2+2, 2k-4+2, \dots,
2k-2(j'-1)+2,2k-2j'+1,\dots, 7, 5, 3)$ of maximum norm. If $j=2k$ then having ``used up'' 
$2k$ of the 1's in the preceding steps, we arrive at partition $\gamma:=(2k+2, 2k, 
2k-2,\dots, 8, 6, 4)$. For $j=2k+1$ we again restart the process for a final move, 
adding 
the remaining 1 to the largest part of $\gamma$.

In each of these cases, the value of the norm is immediate from standard factorial 
manipulations.
 \end{proof}

\begin{remark}
The preceding proof of Theorem~\ref{rrmax} is sufficiently general that it could be used
to prove analogous results for the maximum norm of other restricted classes of 
partitions of size $n$, e.g., 
the \emph{G\"ollnitz--Gordon partitions} (partitions with difference at
least two between parts and no consecutive even numbers as 
parts), or the~\emph{Schur partitions} (partitions
with difference at least three between parts and no consecutive multiples of $3$ as
parts), etc.; see~\cite{Andrews,RR}.
\end{remark} 

Let us now look also at questions of minimality. Clearly the partition of integer $n\geq 
1$ 
of {\it minimum} norm is $(1,1,1,...,1)$, with $n$ repetitions. It is also not hard to see 
that among all partitions of $n\geq 3$ with {\it distinct} parts, the one with 
minimum norm is $(n-1,1)$.

With a slight change of perspective, one might ask instead about partitions of a {\it 
fixed
norm}, say $\nu$, having minimum or maximum {\it size}. The maximal result is easy, 
as any 
number of 1's can be adjoined to a partition without altering its norm; thus there is no 
fixed-norm partition of maximum size.  
The minimum size problem is somewhat less trivial.

\begin{theorem}
The minimum possible size of a partition of norm $\nu$ is  $$a_1 p_1 + a_2 p_2 + a_3 
p_3 + 
\cdots +a_i p_i + \cdots,$$ where 
 $\nu = p_1^{a_1} p_2^{a_2}p_3^{a_3} \cdots  \   p_i^{a_i}\cdots$ is the prime 
factorization 
of $\nu$ ($p_1=2, p_2=3, p_3=5,$ etc., with only finitely many $a_i \geq 0$ being 
nonzero). 
This minimal size is achieved by norm-$\nu$ partitions of the shape
 $$\left<p_1^{a_1-2b} \  p_2^{a_2} \  4^{b}\   p_3^{a_3}\   p_4^{a_4}\dots
 \  p_i^{a_i} \dots \right>$$
 for every integer $b$ such that $0 \leq b \leq \frac{1}{2}a_1$. 
\end{theorem}

\begin{proof}
Consider a partition $\gamma=\left< k_1^{m_{k_1}}\  k_2^{m_{k_2}} \  k_3^{m_{k_3}} ...
\  
k_t^{m_{k_t}}... \right>$ with norm $N(\gamma)=\nu$. We exclude partitions with 1 as a 
part, as some or all of the 1's can be deleted from such a partition, diminishing its size 
without changing its norm. 

Certainly, one partition of norm $\nu$ is $\rho=\left<2^{a_1} \  3^{a_2} \    5^{a_3}\   
7^{a_4}...\  p_i^{a_i} ... \right>$ consisting of the prime factors of $\nu$ including 
multiplicities. For $\gamma \neq \rho$, since the product of the parts of $\gamma$ 
equals $
\nu$, each part is the product of some of the factors of $\nu$, i.e., $k_j=p_1^{c_1} 
p_2^{c_2}p_3^{c_3} \cdots  \   p_i^{c_i}\cdots$ with $0\leq c_i\leq a_i$ for all $i$. Thus 
the parts $k_1, k_2,k_3,...$ essentially represent a regrouping of this set of prime 
factors into a smaller set of numbers including products of some of the primes. 

But since $x_1+x_2+...+x_r\   \leq \  x_1x_2\cdots x_r$ for $x_i \geq 2$ with equality 
only 
in the case $2+2=2\cdot 2$, then $p_1^{c_1} +p_2^{c_2}+p_3^{c_3} + ... \leq k_j$, and 
by 
extension, $|\rho| \leq |\gamma|$. In this case, equality occurs when $\gamma$ is 
formed by 
replacing some number  $b$ of pairs of 2's in partition $\rho$ by the same number $b$ 
of 
4's, since this replacement changes neither the size nor the norm.
\end{proof}

Turning now to asymptotic-type results, we recall work of Lehmer~\cite{L} connecting 
the 
reciprocal of the norm to the {\it Euler--Mascheroni constant} $\gamma = 0.5772 \dots,
$ 
which is defined by $\gamma:=\lim_{n\to\infty} \left( \sum_{k=1}^n \frac{1}{k}  -\log n 
\right)$.
\begin{theorem}[Lehmer] \label{Lehmer1}
We have that
\[ \lim_{n\to\infty} \frac{1}{n} \sum_{\lambda \vdash n} \frac{1}{N(\lambda)} = e^{-
\gamma}.  \]
\end{theorem}

A similar result holds if we restrict the sum to partitions $\mathcal D$ into distinct 
parts.

\begin{theorem}[Lehmer]\label{Lehmer2}   Let $\mathcal{D}$ denote the set of 
partitions 
into distinct parts. Then 
\[ \lim_{n\to\infty} 
\underset{\lambda\in\mathcal{D}}{\sum_{\lambda \vdash n}} \frac{1}{N(\lambda)} = e^{-
\gamma}.  \]
\end{theorem}

We note that the first of the two theorems above is almost-but-not-quite an average 
(the
sum is taken over the partitions of $n$, not over  $1,2,3,...,n$). Along similar lines, it 
is natural to want to know the average magnitude of the norm.

\begin{theorem}\label{expected}
The expected value of the norm over all the partitions of $n$ is 
$$E[N]=\prod_{i=1}^{n}\sqrt[i]{i}.$$
\end{theorem}

\begin{proof}
It is a result of the second author \cite{AVS_MMstats}, which can be proved by setting 
$q_i=1/i$ in Eq. 14 of \cite{Kindt}, that  
the expected value $ E[m_i]$ of the frequency of $i$ obeys 
\begin{equation}\label{stat}
E[m_i ]= \frac{1}{i}.\end{equation}

Then noting $E[ N]=1^{E[m_1 ]}2^{E[m_2 ]}3^{E[m_3 ]} \cdots n^{E[m_n ]}$ completes 
the 
proof. \end{proof}

\begin{remark}
Thus, by \eqref{stat} the expected length of a partition of $n$ is the $n$th harmonic 
number:
$$E[\ell]=E[m_1]+E[m_2]+...+E[m_n]=1+1/2+1/3+...+1/n \sim  \log n + \gamma.$$ 
\end{remark}

Let $\gamma_1 =  -0.0728\dots$ denote the first of the {\it Stieltjes constants} $
\gamma_k,
\   k\geq 0$, 
generalizations of the Euler--Mascheroni constant $\gamma=\gamma_0$ defined by 
the 
coefficients of the Laurent series expansion of the (analytically continued) Riemann 
zeta function $\zeta(s)$ around $s=1$:
\begin{equation}
\zeta(s)=\frac{1}{s-1}+\sum_{k=0}^{\infty}(-1)^{k}(s-1)^k \frac{\gamma_k}{k!}.
\end{equation}
Well-known facts about $\gamma_1$ (see, e.g., \cite{Berndt}) together with Theorem 
\ref{expected} give an asymptotic for the norm.

\begin{corollary}
As $n\to \infty$, the expected value of the norm over partitions of $n$ obeys the 
estimate
$$E[N]\  \sim\   e^{-\gamma_1 n}.$$
\end{corollary}

\begin{remark}
In other words, $\lim_{n\to \infty} \frac{1}{n} \log E[N]=-\gamma_1=0.0728...$.
\end{remark}

It is interesting to see another connection between the partition norm and the Euler--
Mascheroni constant, passing through the Riemann zeta function. In the next section 
we 
explore further zeta function connections. 

\section{Partition zeta functions and analogs of arithmetic functions}
In \cite{Robert_zeta} the first author introduced a broad class of {\it partition zeta 
functions}  
arising from a fusion of Euler's product formulas for both the partition generating 
function and the Riemann zeta function, in which the norm $N(\lambda)$ is the pivotal 
object.   
\begin{definition}\label{ch1pzf}
In analogy to the Riemann zeta function $\zeta(s):=\sum_{n=1}^{\infty}{n^{-s}}$ 
(convergent 
for $\operatorname{Re}(s)>1$), 
for a proper subset $\mathcal P'\subset \mathcal P$ and value $s\in \mathbb C$ for 
which 
the  following series converges, 
we define a {\it partition zeta function} to be the sum over partitions in $\mathcal P'$: 
\begin{equation}\label{ch1zetadef}
\zeta_{\mathcal P'}(s):=\sum_{\lambda\in \mathcal P'}\frac{1}{N(\lambda)^{s}}.
\end{equation}
If we let $\mathcal P'$ equal the partitions $\mathcal P_{\mathbb X}$ whose parts all 
lie 
in 
some subset $\mathbb X \subset \mathbb N,\  1\not\in \mathbb X$, there is also an 
Euler 
product convergent for $\operatorname{Re}(s)>1$:
\begin{equation} \label{ch1zetadef2}
\zeta_{\mathcal P_{\mathbb X}}(s)=\prod_{n\in \mathbb X} \left(1-\frac{1}{n^{s}}
\right)^{-1}. \end{equation}
\end{definition}

Right away, equation \eqref{ch1zetadef2} of Definition \ref{ch1pzf} connects with 
Theorem \ref{loggen}. 

\begin{theorem}[Schneider--Schneider]\label{normthm}
For $\mathbb X \subset \mathbb N, \  1\not\in \mathbb X, \  0< x < e^{-1}, s:=-\log x\in 
\mathbb R^+,$ 
we have
$$\zeta_{\mathcal P_{\mathbb X}}(s)=\prod_{n\in \mathbb X} \frac{1}{1 - x^{\log n}}.$$
\end{theorem}

\begin{proof}
This is an instance of \cite[Remark 4.6]{Sequential}; we flesh out the proof sketched 
there. 
For $x\in \mathbb R$, we have $x^{\log n}=(e^{\log x})^{\log n}=(e^{\log n})^{\log x}
=n^{\log x}=n^{-s}$, giving the product side of the identity. Similarly, we can rewrite $
\zeta(s)=\sum_{n=1}^{\infty} n^{-s}$ as  
\begin{equation} \label{zetalog}
\zeta(-\log x)=\sum_{n=1}^{\infty}x^{\log n}.
 \end{equation} 
Noting $s=-\log x > 1$ for $0<x<e^{-1}$, thus $\zeta(s)=\zeta(-\log x)$ converges 
absolutely, gives convergence in the theorem.
\end{proof}
  
  \begin{remark}  
 Theorem \ref{loggen} is the case $\mathbb X = \mathbb N\setminus \{1\}$ of Theorem 
 \ref{normthm}. 
 \end{remark}

By the Euler product formula for $\zeta(s)$, the usual Riemann zeta function 
represents the partition zeta function $\zeta_{\mathcal P_{\mathbb P}}(s)$, i.e.,  $
\mathbb X = \mathbb P$. 
 
Partition zeta sums over other proper subsets of $\mathcal P$ can yield nice closed-
form 
results of varying natures. To see how choice of subset influences the evaluations, fix 
$s=2$ and sum over three unrelated subsets of $\mathcal P$: partitions $\mathcal 
P_{\mathbb P}$ into prime parts, partitions $\mathcal P_{2\mathbb N}$ into even parts, 
and partitions 
$\mathcal D$ into distinct parts. 
\begin{theorem}[Schneider]
We have the identities
\begin{flalign*}\label{ch1examples}
\zeta_{\mathcal P_{\mathbb P}}(2)=\frac{\pi^2}{6},\  \  \  \  \  \  \  \  \  \  \  \  \  \  \  
\zeta_{\mathcal P_{2\mathbb N}}(2)=\frac{\pi}{2},\  \  \  \  \  \  \  \  \  \  \  \  \  \  \  
\zeta_{\mathcal D}(2)=\frac{\operatorname{sinh} \pi}{\pi}
.\end{flalign*}\end{theorem}
The proofs (see \cite{PhD}) involve variations of Euler's product formula for $\sin x$. 

Notice how different choices of partition subsets induce very different partition zeta 
values for fixed $s$. Interestingly, differing powers of $\pi$ appear in all three examples 
given. 
Another, slightly complicated-looking formula involving $\pi$ arises if we take $s=3$ 
(noting the value of the case $\zeta(3)$ is unknown) and sum on nuclear partitions 
defined above. 

Let us denote the set of nuclear partitions (partitions with no $1$'s)
 by $\mathcal N$, and recall $\widetilde{P}(\nu)$ enumerates nuclear partitions of norm 
 $\nu$, i.e., {multiplicative partitions} of $\nu$.

\begin{theorem}
We have that 
$$\zeta_{\mathcal N}(3)=\sum_{\nu=1}^{\infty}\frac{\widetilde{P}(\nu)}{\nu^3}=\frac{3\pi}
{\operatorname{cosh}\left(\frac{1}{2}\pi \sqrt{3}\right)}.$$
\end{theorem}

\begin{proof}
That the partition zeta function equals the right-hand value is \cite[Corollary 2.3]
{Robert_zeta}. Now, using Theorem \ref{normthm} on the left side of Theorem 
\ref{loggen}, and noting that  $s=-\log x$ gives $x^{\log \nu}=\nu^{-s}$ on the right side, 
yields $\zeta_{\mathcal N}(s)=\sum_{\nu=1}^{\infty}\widetilde{P}(\nu)\nu^{-s};$ setting 
$s= 3$ (i.e., $x=e^{-3}$) completes the proof.\end{proof}

These partition formulas for $\pi$ are interesting, but they look a little too disparate to 
comprise a {\it family} like Euler's values $\zeta(2k)= \pi^{2k}\times$ ``rational number''. 
There is at least one (non-Riemann) class of partition zeta functions that yields nice 
evaluations like this. 
\begin{definition} We define
$$\zeta_{\mathcal P}(\{s\}^k):=\sum_{\ell(\lambda)=k}\frac{1}{N(\lambda)^s},$$
with the sum taken over all partitions of fixed length $k\geq 0$, with $\zeta_{\mathcal 
P}(\{s\}^0):=N(\emptyset)^{-s}=1$.
\end{definition}

The $k=1$ case is $\zeta(s)$. At argument $s=2$ these partition zeta functions yield 
explicit values closely related to Euler's even-argument zeta evaluations. 
\begin{theorem}[Schneider]\label{zeta}
For $s=2, \  k\geq 1$, we have the identity 
\begin{equation*}
\zeta_{\mathcal P}(\{2\}^k) = \frac{2^{2k - 1} - 1}{2^{2k-2}}\zeta(2k),
\end{equation*}
and analogous formulas exist for partitions into distinct parts.
\end{theorem}

So these particular partition zeta values are rational multiples of Euler's zeta values 
(and of $\pi^{2k}$). Note that if we set $k=0$ and solve the above identity for $
\zeta(0)$, we conclude formally that $\zeta(0) = \frac{2^{-2}}{2^{-1} - 1}\zeta_{\mathcal 
P}(\{2\}^0) = -1/2$, which is the correct value for $\zeta(0)$ under analytic continuation. 
This raises the question of analytic continuation for the function $\zeta_{\mathcal P}(\{s
\}^k)$.

The preceding zeta formulas and numerous others, including general structural 
relations, are proved in \cite{Robert_zeta}. In \cite{ORS}, the authors prove other facts 
about partition zeta functions, including a farther-reaching follow-up to Theorem 
\ref{zeta}.\begin{theorem} [Ono--Rolen--Schneider]
For $m\geq 1, k\geq 1$, we have
$$\zeta_{\mathcal P}(\{2m\}^k)=\pi^{2mk}\times \text{``rational number''}.$$ 
\end{theorem}
These zeta sums over partitions of fixed length do indeed form a family like Euler's 
zeta 
values for positive even arguments. In \cite{SSzeta}, we 
give a closed formula for general $s\in \mathbb C$ as a combination of Riemann zeta 
functions and MacMahon coefficients, via both analytic and combinatorial proofs. 
\begin{theorem}[Schneider--Sills]\label{Pennthm}
For $\operatorname{Re}(s)>1, k \geq 1$, we have  
$$\zeta_{\mathcal P}(\{s\}^k)=\sum_{\lambda \vdash k}
\frac{\zeta(s)^{m_1}\zeta(2s)^{m_2}\zeta(3s)^{m_3}\cdots \zeta(ks)^{m_k}}{N(\lambda)\  
m_1!\  
m_2!\  m_3!\  \cdots\  m_k!}.$$
\end{theorem}

Thus $\zeta_{\mathcal P}(\{s\}^k)$ inherits analytic continuation as well as trivial zeroes 
at $s=-2, -4, -6, ...$, from $\zeta(s)$, has poles at $s=1, 1/2,1/3,1/4,...,1/k$ with the 
order of the pole $s=1/i$ being $\left\lfloor k/i \right\rfloor, 1\leq i \leq k$, and indeed 
equals $\pi^{2mk}\times$ ``rational number'' for $s=2m, m\geq 1$.

Zeta functions are only the trail-head of many paths connecting partition theory and 
classical multiplicative number theory, as shown for instance by the first author and his 
collaborators in \cite{ORS, OSW, Robert_zeta, Robert_bracket, PhD}. In addition to the 
zeta 
function analogs seen already, there are partition-theoretic versions of classical 
arithmetic functions such as the M\"{o}bius function $\mu(n)$, the sum of divisors 
function 
$\sigma(n)$, the Euler phi function $\varphi(n)$, etc.  
{\it Partition Dirichlet series} are also defined for any function $f:\mathcal P \to \mathbb 
C$ defined on partitions (see \cite{ORS, PhD}), viz. for $\mathcal P' \subseteq \mathcal 
P$ 
we set
\begin{equation}
\mathscr{D}_{\mathcal P'}(f,s):=\sum_{\lambda \in \mathcal P'}\frac{f(\lambda)}
{N(\lambda)^s}
\end{equation}
where convergence depends on $f$ and $s\in \mathbb C$ as well as the subset $
\mathcal P'$. 

To give a concrete example, the {\it partition phi function} $\varphi_{\mathcal P}
(\lambda)$ is defined in \cite{Robert_bracket} in terms of the norm:
\begin{equation}
\varphi_{\mathcal P}(\lambda):=N(\lambda)\prod_{\substack{\lambda_i \in \lambda\\ 
\text{no repetition}}}\left(1-\frac{1}{\lambda_i}\right),
\end{equation}
where the product is taken over the parts $\lambda_i$ of $\lambda$ {without} 
repetition. This function fits into partition theory in an almost identical manner to $
\varphi(n)$ in elementary number theory, as the following pair of identities suggests.

For $\delta, \lambda \in \mathcal P$, we say $\delta$ is a {\it subpartition} of $\lambda$ 
and write ``$\delta|\lambda$'' if all the parts of $\delta$ are also parts of $\lambda$ 
including their frequencies.
\begin{theorem} [Schneider]\label{phi}
We have the following identities:
\begin{equation*}
\sum_{\delta | \lambda} \varphi_{\mathcal P}(\delta) = N(\lambda)
,\  \  \  \  \  \  \  \  \  \  \  \  \  \  \  \sum_{\lambda\in\mathcal P_{\mathbb X}}
\frac{\varphi_{\mathcal P}(\lambda)} {N(\lambda)^s}=\frac{\zeta_{\mathcal P_{\mathbb 
X}}
(s-1)}{\zeta_{\mathcal P_{\mathbb X}}(s)}\  \   \  (\operatorname{Re}(s)>2),
\end{equation*}
where in the first sum, ``$\delta | \lambda$'' means the sum is taken over subpartitions 
of 
$\lambda$, and the second sum holds for any subset $\mathbb X \subset \mathbb N$.
\end{theorem}
The second summation above represents a partition Dirichlet series, and actually 
converges 
even for $\mathbb X = \mathbb N$ since $\varphi_{\mathcal P}$ vanishes on partitions 
with 
any part $=1$ (but the ratio of zeta functions on the right-hand side then becomes 
indeterminate since $\zeta_{\mathcal P}(z)$ is infinite for all $z\in\mathbb C$). These 
formulas generalize the classical identities
\begin{equation}
\sum_{d|n}\varphi(n) = n,\  \  \  \  \  \  \  \  \  \  \  \  \  \  \  \  \sum_{n=1}^{\infty}
\frac{\varphi(n)}{n^s}=\frac{\zeta(s-1)}{\zeta(s)}\  \   \  (\operatorname{Re}(s)>2).
\end{equation}
 
Other well-known objects and identities from multiplicative number theory also 
represent special cases of partition-theoretic theorems (see \cite{OSW, 
Robert_bracket, PhD} for further reading). 

We close with a curious identity connecting the nice family of partition zeta functions 
described in Theorem \ref{zeta} to another constant of much interest historically, as 
well as $\pi$.
\begin{theorem}\label{golden} 
Let $\phi=\frac{1+\sqrt{5}}{2}$ denote the golden ratio. Then we have 
\begin{equation*}
\frac{\phi\  \pi}{5}\  =\  \sum_{k=0}^{\infty}\frac{\zeta_{\mathcal P}(\{2\}^k)}{100^k}.
\end{equation*}
\end{theorem}
\begin{proof}
The equation results from comparing Theorem \ref{Pennthm} above with the 
coefficients of 
$1/100^k$ in the first identity of \cite[Proposition D.2.4]{PhD}:
\begin{equation}\label{Dthm1.2}
\phi\  =\  \frac{5}{\pi}\sum_{\lambda\in \mathcal P } 
\frac{\zeta(2)^{m_1}\zeta(4)^{m_2}\zeta(6)^{m_3}\zeta(8)^{m_4}\cdots}{\  N(\lambda)\   
100^{|\lambda|}\  m_1!\  m_2!\   m_3!\   m_4!\  \cdots},
\end{equation}
which itself follows from trigonometric facts about the golden ratio, Euler's product 
formula for the sine function, the Maclaurin series for $-\text{log}(1-x)$ 
and Fa\`{a} di Bruno's formula.
\end{proof}
 
\begin{remark}

We note that by Theorem \ref{zeta}, the right-hand sum of Theorem \ref{golden} may 
be rewritten in terms of $\zeta(2k)$. 
\end{remark}

Due to the tantalizing connections we find it to have in the literature, as well as in our 
research, the partition norm seems worthy of further study in its own right.

\subsection*{Acknowledgments} The authors are grateful to the organizers of the 
Integers Conference 2018, and to the anonymous referee for useful suggestions.

\end{document}